\documentclass[10pt,a4paper]{amsart}
\usepackage{amsmath,amsthm,amssymb,latexsym,enumerate,color,hyperref}
\usepackage[applemac]{inputenc}
\usepackage{graphicx}

\numberwithin{equation}{section}

\newtheorem{thm}{Theorem}[section]
\newtheorem{prop}[thm]{Proposition}
\newtheorem{lem}[thm]{Lemma}
\newtheorem{cor}[thm]{Corollary}
\newtheorem{rem}[thm]{Remark}

\DeclareMathOperator*{\med}{med} 
\DeclareMathOperator*{\pmed}{h-med} 
\DeclareMathOperator*{\esssup}{ess\,sup} 
\DeclareMathOperator*{\essinf}{ess\,inf} 

\def\Xint#1{\mathchoice
{\XXint\displaystyle\textstyle{#1}}%
{\XXint\textstyle\scriptstyle{#1}}%
{\XXint\scriptstyle\scriptscriptstyle{#1}}%
{\XXint\scriptscriptstyle\scriptscriptstyle{#1}}%
\!\int}
\def\XXint#1#2#3{{\setbox0=\hbox{$#1{#2#3}{\int}$}
\vcenter{\hbox{$#2#3$}}\kern-.5\wd0}}

\def\dashint{\Xint-}

\newcommand{\myqed}{\hspace*{\fill}\hbox{\rule[-2pt]{3pt}{6pt}}}

\newcommand{\link}{\mathop{\circ\kern-.35em -}}

\newcommand{\PA}{\partial}
\newcommand{\ve}{\varepsilon}

\newcommand{\ol}{\overline}

\newcommand{\la}{\lambda}
\newcommand{\La}{\Lambda}    
\newcommand{\de}{\delta}

\newcommand{\lan}{\langle}
\newcommand{\ran}{\rangle}
\newcommand{\tr}{\mathop{\mathrm{tr}}}

\newcommand{\al}{\alpha}
\newcommand{\be}{\beta}
\newcommand{\Om}{\Omega}
\newcommand{\na}{\nabla}

\newcommand{\nr}{\Vert}
\newcommand{\NN}{\mathbb{N}}
\newcommand{\RR}{\mathbb{R}}
\renewcommand{\SS}{\mathbb{S}}
\newcommand{\De}{\Delta}
\newcommand{\cX}{\mathcal{X}}

\newcommand{\om}{\omega}
\newcommand{\si}{\sigma}
\newcommand{\te}{\theta}
\newcommand{\Ga}{\Gamma}

\numberwithin{equation}{section}

\title[Asymptotic mean value property for the $p$-Laplacian]{\bf A natural approach to the \\ asymptotic mean value property \\ for the $p$-Laplacian}
\author[M. Ishiwata]{Michinori Ishiwata}
\address{Department of Systems Innovation
Graduate School of Engineering Science, Osaka University,
1-3 Machikaneyama, Toyonaka, Osaka 560-8531, Japan}
\email{ishiwata@sigmath.es.osaka-u.ac.jp}

\author[R. Magnanini]{Rolando Magnanini}
\address{Dipartimento di Matematica ``U. Dini'', Universit\`a a di Firenze, viale Morgagni 67/A, 50134 Firenze, Italy}
\email{magnanin@math.unifi.it, paolo.salani@math.unifi.it}
\urladdr{http://web.math.unifi.it/users/magnanin}

\author[H. Wadade]{H. Wadade}
\address{Faculty of Mechanical Engineering, Institute of Science and Engineering, Kanazawa University,
Kakuma, Kanazawa, Ishikawa 920-1192, Japan}
\email{wadade@se.kanazawa-u.ac.jp}


\begin{document}

\begin{abstract}
Let $1\le p\le\infty$. We show 
that 
a function $u\in C(\mathbb R^N)$ is a viscosity solution to 
the normalized $p$-Laplace equation $\Delta_p^n u(x)=0$ if and only if the asymptotic formula 
$$
u(x)=\mu_p(\ve,u)(x)+o(\ve^2)
$$
holds as $\ve\to 0$ in the viscosity sense. Here, $\mu_p(\ve,u)(x)$ is 
the $p$-mean value of $u$ on $B_\ve(x)$ characterized as a unique minimizer of 
$$
\inf_{\la\in\RR}\nr u-\la\nr_{L^p(B_\ve(x))}. 
$$
This kind of asymptotic mean value property (AMVP) extends to the case $p=1$ previous (AMVP)'s obtained when $\mu_p(\ve,u)(x)$ is replaced by other kinds of mean values.  The natural definition of $\mu_p(\ve,u)(x)$ makes sure that this is a monotonic and continuous (in the appropriate topology) functional of $u$. These two properties help to establish a fairly general proof of (AMVP), that can also be extended to the (normalized) parabolic $p$-Laplace equation.
\end{abstract}

\maketitle

\section{Introduction and main theorems} 

It is well-known that the classical {\it mean value property} characterizes harmonic functions and helps to derive most of their salient properties, such as weak and strong maximum principles, analyticity, Liouville's theorem, Harnack's inequality and more.  In fact, we know that
a continuous function $u$ is harmonic in an open set $\Om\subseteq\mathbb R^N$ if and only if 
\begin{equation}\label{harmo-stat}
u(x)=\dashint_{B_\ve(x)}u(y)\,dy=\dashint_{\PA B_\ve(x)}u(y)\,dS_y
\end{equation}
for every ball $B_\ve(x)$ with $\ol{B_\ve(x)}\subset\Om$; 
here, $\dashint_E u$ denotes the {\it mean value} of $u$ over a set $E$ with respect to the relevant measure 
(see Evans \cite{E} for instance). 
The relation \eqref{harmo-stat} can also be regarded as a statistical characterization of solutions of the Laplace equation, without an explicit appearance of derivatives of $u$. 
A similar mean value property can also be obtained for linear elliptic equations with constant coefficients, by replacing balls by appropriate ellipsoids (see \cite{BL}
and \cite{BLU}). 
\par
Recently, starting with the work \cite{MPR2} of Manfredi, Parviainen and Rossi, a great attention has been paid to the so-called {\it asymptotic mean value property (AMVP)} and its applications to game theory.
In \cite{MPR2}, based on the formula
\begin{eqnarray}
\dashint_{B_\ve(x)}v(y)\,dy=v(x)+\frac12\,\frac{\De v(x)}{N+2}\,\ve^2+
o(\ve^2)\quad\text{as \,}\ve\to 0,
\label{eq:hf02}
\end{eqnarray}
that holds for any smooth function $v$ not necessarily harmonic,
it is shown that the characterization \eqref{harmo-stat} for the
harmonicity of $u$ can be replaced by the weaker (AMVP):
\begin{equation}\label{harmo-relax}
u(x)=\dashint_{B_\ve(x)}u(y)\,dy+o(\ve^2)\quad\text{as \,}\ve\to 0,
\end{equation}
for all $x\in\Om$.  
\par
Nonetheless, the decisive contribution of \cite{MPR2} is the observation  that,
provided the mean value in \eqref{harmo-relax} is replaced by a suitable (nonlinear) statistical value related to $u$, 
an (AMVP) also characterizes {\it $p$-harmonic} functions, that is the (viscosity) solutions of the {\it normalized $p$-Laplace equation} $\De^n_p u=0$. Here,
\begin{eqnarray*}
&\De^n_p u=\displaystyle\frac{\nabla\cdot(|\nabla u|^{p-2}\nabla u)}{|\nabla u|^{p-2}} \ \mbox{ for } \ 1\le p<\infty,
\ \De^n_\infty u=\displaystyle\frac{\lan\na^2v\,\na v,\na v\ran}{|\na v|^2},
\end{eqnarray*}
denotes the so-called {\it normalized or homogeneous $p$-Laplacian}.
\par
In fact, in the same spirit of \eqref{eq:hf02}, for $1<p\le\infty$ and for any smooth function with $\na v(x)\not=0$,
they proved the formula:
$$
\mu_p^*(\ve,v)=v(x)+\frac12\,\frac{\Delta^n_p v(x)}{N+p}\,\ve^2+
o(\ve^2) \ \mbox{ as } \ \ve\to 0,
$$
where
\begin{equation}
\label{mean-mpr}
\mu^*_p(\ve,u)=\frac{N+2}{N+p}\,\dashint_{B_\ve(x)}u(y)\,dy+\\
\frac12\,\frac{p-2}{N+p}\left(
\max_{\ol{B_\ve(x)}}u+\min_{\ol{B_\ve(x)}}u\right).
\end{equation}
(The average of the minimum and the maximum will be referred to as the {\it min-max mean} of $u$.)
\par
That formula allowed them to prove that
$u$ is $p$-harmonic in the viscosity sense in $\Om$ if and only if
\begin{equation}
\label{eq:hf042}
u(x)=\mu_p^*(\ve,u)+
o(\ve^2)\quad\text{as \,}\ve\to 0,
\end{equation}
in the viscosity sense for every $x\in\Om$ (see Section \ref{sec:elliptic} for the relevant definitions), thus obtaining an (AMVP) for $p$-harmonic functions. It is also worth a mention that, for $N=2$ and small values of the parameter $p>1$, in \cite{LM} it is proved that the (AMVP) holds directly for weak
solutions of the $p$-Laplace equation, without the need to interpret the formula in the viscosity sense.
\par
Thus, the mean $\mu^*_p(\ve,u)$ is an example of the desired (nonlinear) statistical value mentioned above. By similar arguments, one can obtain an (AMVP) with the ball $B_\ve(x)$ replaced by the sphere $\PA B_\ve(x)$ simply by replacing $(N+2)/(N+p)$ and $(p-2)/(N+p)$ by the numbers
$N/(N+p-2)$ and $(p-2)/(N+p-2)$. 

\medskip

In the quest of extending this type of result to the case $p=1$, which is not covered by the choice \eqref{mean-mpr}, other kinds of means were proposed by several authors.
Here, we mention the ones considered by Hartenstine and Rudd in \cite{HR}, based on the {\it median}
of a function, \par
\begin{eqnarray}
\mu_{p}'(\ve,u)=\frac{1}{p}\,\med_{\PA B_\ve(x)} u+\frac{p-1}{2p}\left(
\min_{\ol{B_\ve(x)}}u+\max_{\ol{B_\ve(x)}}u\right), \label{svh}\\
\mu_p''(\ve,u)=\frac{2-p}{p}\,\med_{\PA B_\ve(x)} u+\frac{2(p-1)}{p}\,\dashint_{\PA B_\ve(x)}u(y)\,dS_y, \label{svr}
\end{eqnarray}
and that considered by  Kawohl, Manfredi and Parviainen in \cite{KMP},
\begin{equation}
\label{mean-kmp}
\mu_p^{**}(\ve,u)=\frac{N+1}{N+p}\,\operatorname{av}_\ve(u)(x)
+\frac12\,\frac{p-1}{N+p}\left(\min_{\overline{B_\ve(x)}}u+\max_{\overline{B_\ve(x)}}u\right),
\end{equation}
where
\begin{eqnarray*}
\label{def-av}
\operatorname{av}_\ve(u)(x)=\dashint_{L_\ve}u(x+y)\,dS_y, 
\end{eqnarray*}
$L_\ve=\{ y\in B_\ve(x): (y-x)\cdot\nu=0\}$ and
$$
 \nu=\nu_{x, \ve}\in\PA B_1(0) \ \mbox{ is such that } \ u(x+\ve\nu)=\min_{\overline{B_\ve(x)}}u.
$$
\par
Both $\mu_{p}'(\ve,u)$ and $\mu_{p}''(\ve,u)$ yield an (AMVP) for all the cases $1\le p\le\infty$, but {\it only when $N=2$}, and $\mu_{p}^{**}(\ve,u)$ produces an (AMVP) for any $1\le p\le\infty$ and $N\ge 2$.

\medskip

In this paper, for $1\le p\le\infty$, we propose one more mean that helps us to characterize --- in an intrinsic way --- $p$-harmonic functions by an (AMVP). Its definition was inspired by the simple remark that the median, the mean value and the min-max mean of a continuous function $u$
on a compact topological space $X$ equipped with a positive Radon measure $\nu$ 
are respectively the unique real values $\mu_p^X(u)$ that solve the variational problem
\begin{equation}
\label{var-probl}
\nr u-\mu_p^X(u)\nr_{L^p(X,\nu)}=\min_{\la\in\RR}\,\nr u-\la\nr_{L^p(X,\nu)},
\end{equation}
for $p=1, 2$, and $\infty$. Thus, it is natural to ask whether the solution of \eqref{var-probl} yields a characterization of viscosity solutions of $\De_p^n u=0$ by 
means of an (AMVP), \textit{for each fixed $1\le p\le\infty$}.
\par
Therefore, for each $1\le p\le\infty$,  we consider the {\it $p$-mean of $u$} in $B_\ve(x)$, that is the number defined as
\begin{equation}
\label{p-mean}
\mu_p(\ve,u)(x)=\mbox{the unique $\mu\in\RR$ satisfying \eqref{var-probl} with } X=\ol{B_\ve(x)}.
\end{equation}
The main result of this paper is the following characterization.
\begin{thm}
\label{th:main:visco} 
Let $1\le p\le\infty$ and let $\Om$ be an open subset of $\RR^N$.
For a function $u\in C(\Om)$ the following assertions are equivalent:
\begin{enumerate}[(i)]
\item
$u$ is a viscosity solution of $\Delta_p^n u=0$ in $\Om$;
\item
$u(x)=\mu_p(\ve,u)(x)+o(\ve^2)$ as $\ve\to 0$, in the viscosity sense for every $x\in\Om$. 
\end{enumerate}
\end{thm}
\noindent
As a by-product, this theorem confirms the (AMVP) for $\mu'_p(\ve,u)$ and $\mu^*_p(\ve,u)$ for the case $p=1$ in {\it any dimension} $N\ge 2$.  
\par
Theorem \ref{th:main:visco} is based on the asymptotic formula
$$
\mu_p(\ve,v)(x)=v(x)+\frac12\,\frac{\De_p^n v(x)}{N+p}\,\ve^2+o(\ve^2) \ \mbox{ as } \ \ve\to 0,
$$
that holds for any smooth function $v$ such that $\na v(x)\not=0$.
\par
\par
We mention in passing that the mean $\mu_p^X(u)$ has also be considered in \cite{FFGS}, when $p\ge 2$ and $N=2$, 
when $X$ is a finite set and
$\nu$ is the counting measure  and has proved  to be effective in the numerical approximation of the operator $\De_p^n$.  Another type of (AMVP) has been proved in \cite{GS} for $N=2$ and $1<p<\infty$;
 however, the mean considered there, besides the values of the function $u$ on $B_\ve(x)$, also depends on the value of $\na u$ at $x$. 
\par
Compared to the means defined in \eqref{mean-mpr}, \eqref{svh}, \eqref{svr}, and \eqref{mean-kmp} (and that in \cite{GS}),
$\mu_p(\ve,u)$ has a drawback, since it cannot be defined explicitly, unless $p=1, 2, \infty$. However,
it has useful properties that those means do not always have and are the consequences of the fact that 
$\mu_p(\ve,u)$ is the {\it projection} of $u$ on the linear sub-space of $L^p(B_\ve(x))$
of the constant functions. As a matter of fact, we shall show that the functional 
$L^p(B_\ve(x))\ni u\mapsto\mu_p(\ve,u)(x)\in\RR$ is {\it continuous} in the corresponding $L^p$-topology
and {\it monotonic}, in the sense that 
$$
u\le v \mbox{ pointwise implies that } \mu_p(\ve,u)(x)\le\mu_p(\ve,v)(x). 
$$
Notice that the functionals defined by $\mu_p'\ve,u)$ and $\mu_p^{**}(\ve,u)$ are always monotonic, but
never continuous for $p\in (1,\infty)\setminus\{ 2\}$, while those defined by $\mu_p^*(\ve,u)$ and $\mu_p''(\ve,u)$ are not always monotonic (the former for $p>2$, the latter for $1<p<2$) and
never continuous for $p\in (1,\infty)\setminus\{ 2\}$, due to the presence of the min-max mean
in their definition.
\par
We shall see that the properties of continuity and monotonicity play an essential role in the proof of 
Theorem \ref{th:main:visco}, since they allow to reduce the argument to the simpler case of a {\it quadratic polynomial} (see Lemma \ref{lm:qp} and Theorem \ref{thm:p=p}).

\medskip

With a few technical adjustments, it is not difficult to treat the case of the parabolic $p$-Laplace operator. It is just the matter of replacing the euclidean ball and the Lebesgue measure by a suitable measure space. The appropriate choice is  the so-called {\it heat ball},
$$
E_\ve(x,t)=\left\{ (y,s)\in\RR^{N+1}: s<t,\, \Phi(x-y,t-s)>\ve^{-N}\right\},
$$
where 
$$
\Phi(y,s)=
(4\pi s)^{-N/2} e^{-\frac{|y|^2}{4s}} \,\cX_{(0,\infty)}(s) \ \mbox{ for } \ (y,s)\in\RR^N\times(-\infty, \infty)
$$
is the fundamental solution for the heat equation, equipped with the space-time measure
$$
d\nu(y,s)=\frac{|x-y|^2}{(t-s)^2}\,dy\,ds.
$$
\par
Thus, by arguing in a similar spirit, we  shall consider the value $\pi_p(\ve,u)(x,t)$ as the unique solution  of the var\-i\-a\-tion\-al problem
\begin{equation}
\label{var-probl-heat}
\nr u-\pi_p(\ve,u)(x,t)\nr_{L^p(E_\ve(x,t), \nu)}=
\min_{\la\in\RR}\,
\nr u-\la\nr_{L^p(E_\ve(x,t), \nu)},
\end{equation}
Notice that the value $\pi_p(\ve,u)(x,t)$ in \eqref{var-probl-heat} can be easily computed for $p=2$ as the {\it caloric mean value} of $u$, for which a classical mean value property holds true for solutions of the heat equation (\cite{E}[pp. 52-54]). If we define the space-time cylinder $\Om_T=\Om\times(0,T)$, 
we can prove the following companion of Theorem \ref{th:main:visco} .

\begin{thm}
\label{th:pa-visco} 
Let $1\le p\le\infty$. For a function $u\in C(\Om_T)$, the following assertions are equivalent:
\begin{enumerate}[(i)]
\item
$u_t=\frac{N}{N+p-2}\,\Delta_p^n u$  in $\Om_T$ in the viscosity sense;
\item
$u(x,t)=\pi_p(\ve,u)(x,t)+o(\ve^2)$ as $\ve\to 0$ in the viscosity sense
for every $(x,t)\in\Om_T$. 
\end{enumerate}
\end{thm}

\medskip

For further developments and applications of  (AMVP)'s, we refer
the reader to \cite{AL}, \cite{JLM}, \cite{LM}, \cite{LPS},  
\cite{MPR3}, \cite{RV},  and references therein. 

\medskip

This paper is organized as follows. In Section \ref{sec:properties}, 
we derive the pertinent properties of the $p$-mean value of a continuous function $u$: continuity and monotonicity will be the most important.  
Then, we shall prove Theorem \ref{th:main:visco} and Theorem \ref{th:pa-visco}  in Sections \ref{sec:elliptic} and \ref{sec:parabolic}, respectively. Finally, Section \ref{sec:computations} is
devoted to the calculation of some relevant integrals.

\section{Properties of $p$-mean values}
\label{sec:properties}

Let $X$ be a compact topological space which is also a measure space with respect to a positive Radon measure $\nu$ such that $\nu(X)<\infty$.  
We recall that, if $u\in C(X)$, the {\it median} $\med\limits_X u$ of $u$ in $X$ is defined as the unique solution $\la$ of the equation
\begin{equation}
\label{def-med}
\nu(\{ y\in X: u(y)\ge\la\})=\nu(\{ y\in X: u(y)\le\la\}).
\end{equation}
 
We start by showing that the definitions \eqref{p-mean} and \eqref{var-probl-heat} of $\mu_p(\ve, u)$ and $\pi_p(\ve,u)$ are well posed.

\begin{thm}
\label{thm:characterization}
Let $1\le p\le\infty$ and $u\in C(X)$. There exists a unique real value $\mu^X_p(u)$ such that
$$
\nr u-\mu^X_p(u)\nr_{L^p(X),\nu}=\min_{\la\in\RR}\nr u-\la\nr_{L^p(X,\nu)}.
$$
In particular, 
\begin{eqnarray*}
&\displaystyle \mu_1^X(u)=\med_X u, \qquad \mu_2^X(u)=\dashint_X u(y)\,d\nu \\ 
&\displaystyle \mbox{ and } \ \mu_\infty^X(u)=
\frac12\left(\min_X u+\max_X u\right).
\end{eqnarray*}
\par
Furthermore, for $1\le p<\infty$, $\mu_p^X(u)$ is characterized by the equation
\begin{equation}
\label{eq:characterization}
\int_X \left|u(y)-\mu_p^X(u)\right|^{p-2}\left[u(y)-\mu_p^X(u)\right]\,d\nu=0,
\end{equation}
where, for $1\le p<2$, we mean that the integrand is zero if $u(y)-\mu_p^X(u)=0$.
\end{thm}

\begin{proof} 
The case $p=1$ is a straightforward extension of the proofs in \cite{N}, \cite{Sh} and \cite{S}.
\par
If $p=\infty$, the assertion follows at once by observing that
$$
\max_X|u-\la|=\max\left(\max_X u-\la,\,\la-\min_X u\right). 
$$
\par
Next, in the case $1<p<\infty$, we observe that
$$
\min_{\la\in\RR}\nr u-\la\nr_{L^p(X,\nu)}=\min_{v\in\La}\nr u-v\nr_{L^p(X,\nu)},
$$
where $\La$ is the subspace of constant functions on $X$;
in other words $\mu_p^X(u)$ is a projection of $u$ on $\La$.
Thus, the existence, uniqueness and characterization of $\mu_p^X(u)$ are guaranteed by the theorem of the projection, since $L^p(X,\nu)$ is uniformly convex and $\La$ is a closed subspace, and the differentiability of the function $\la\mapsto \nr u-\la\nr_{L^p(X,\nu)}$ (see \cite{LL}).
\par
The expression of $\mu_2^X(u)$ is readily computed as the minimum point of a quadratic polynomial.
\end{proof}

\begin{rem}
\rm 
Note that, for $1<p\le\infty$, Theorem \ref{thm:characterization} extends to the case in which $u\in L^p(X,\nu)$, provided the minimum and the maximum are replaced by 
$$
\essinf_{X} u \ \mbox{ and } \ \esssup_{X} u.
$$
\par
If $u\in L^1(X,\nu)\setminus C(X)$, it is known that the median of $u$ in $X$ may not be unique (see \cite{N}). 
\end{rem}

\medskip

The following corollary will be very useful for further computations. We set $B=B_1(0)$. 

\begin{cor}
\label{cor:characterization} 
Let $u\in L^p(B_\ve(x))$, for $1< p\le\infty$, and $u\in C(B_\ve(x))$, for $p=1$. 
\par
If we let $u_{\ve}(z)=u(x+\ve z)$ for $z\in\ol{B}$ and set
\begin{equation}
\label{def-mu}
\mu_p(\ve,u)(x)=\mu_p^{B_\ve(x)}(u),
\end{equation}
then it holds that
\begin{equation}
\label{rescaling}
\mu_p(\ve,u)(x)=\mu_p(1,u_{\ve})(0).
\end{equation}
\end{cor} 
\begin{proof}
It suffices to observe that, for every $\la\in\RR$, it holds that
$$
\nr u-\la\nr_{L^p(B_\ve(x))}=\ve^{N/p}\,\nr u_\ve-\la\nr_{L^p(B)},
$$
for $1\le p<\infty$, and 
$$
\nr u-\la\nr_{L^\infty(B_\ve(x))}=\nr u_\ve-\la\nr_{L^\infty(B)},
$$
and hence invoke the uniqueness part of Theorem \ref{thm:characterization}.
\end{proof}

In the next two theorems we regard $\mu_p^X(u)$ as the value at $u$ of a functional $\mu_p^X$ on
$L^p(X)$. If $p=1$ and $u\in L^1(X)\setminus C(X)$, we allow $\mu_1^X(u)$ to be any minimizing value 
of $\la\mapsto\nr u-\la\nr_{L^1(X)}$ on $\RR$, whenever it is convenient.
\begin{thm}[Continuity]
\label{thm:continuity}
Let $1\le p\le\infty$. It holds that  
\begin{equation}
\label{triangle}
\left|\left\nr u-\mu^X_p(u)\right\nr_{L^p(X)}-\left\nr v-\mu^X_p(v)\right\nr_{L^p(X)}\right|\le \nr u-v\nr_{L^p(X)},
\end{equation}
for any $u, v\in L^p(X)$.
\par
Moreover, if $u_n\to u$ in $L^p(X)$ for $1\le p\le\infty$ and $u_n, u\in C(X)$ for $p=1$, then $\mu^X_p(u_n)\to \mu^X_p(u)$ as $n\to\infty$.  
\par
In particular, the same conclusion holds for any $p\in[1,\infty]$, if $\{u_n\}_{n\in\NN}\subset C(X)$ converges to $u$ uniformly on $X$ as $n\to\infty$.
\end{thm}
\begin{proof}
The inequality \eqref{triangle} simply follows by observing that 
$
\left\nr u-\mu^X_p(u)\right\nr_{L^p(X)}
$
is nothing else than the distance of $u$ from the subspace $\La$.
\par
Next, if  $u_n\to u$ in $L^p(X)$ as $n\to\infty$, \eqref{triangle} implies that
$$
\nr u_n-\mu_p^X(u_n)\nr_{L^p(X)}\to \nr u-\mu_p^X(u)\nr_{L^p(X)} \ \mbox{ as } \ n\to\infty.
$$
We conclude by observing that, since $\mu^X_p(u)$ unique for $1<p\le\infty$ and for $p=1$ if $u\in C(X)$, any converging sub-sequence of $\{\mu^X_p(u_n)\}_{n\in\NN}$ must 
converge to $\mu^X_p(u)$.
\end{proof}

\begin{thm}[Monotonicity]
\label{thm:monotonicity}
Let $u$ and $v$ be two functions in $L^p(X)$, for $1<p\le\infty$, or in $C(X)$ for $p=1$. 
\par
If $u\le v$ a.e. on $L^p(X)$, then 
$\mu^X_p(u)\le\mu^X_p(v)$.
\end{thm}
\begin{proof}
For $1<p<\infty$, we observe that the function $F:\RR\times\RR\to \RR$ defined by
$$
F(u,\la)=|u-\la|^{p-2} (u-\la)
$$
is increasing in $u$ for fixed $\la$ and decreasing in $\la$ for fixed $u$; hence,
$$
\int_X |u(y)-\la|^{p-2} [u(y)-\la]\,d\nu_y\le\int_X |v(y)-\la|^{p-2} [v(y)-\la]\,d\nu_y
$$
if $u\le v$ a.e. in $X$.  The characterization \eqref{eq:characterization} then yields that 
$\mu^X_p(u)\le\mu_p^X(v)$.
\par
Next, we know that
$$
\mu^X_\infty(u)=\frac12\left\{\essinf_X u+\esssup_X u\right\};
$$
thus, the conclusion follows by an inspection.
\par
Finally, we know that $\mu^X_1(u)$ is the unique zero of the function defined by
\begin{equation*}
\label{def-F_u}
F_u(\la)=\nu(\{ y\in X: u(y)\ge\la\})-\nu(\{ y\in X: u(y)\le\la\}), \ \la\in\RR.
\end{equation*}
The conclusion then follows by observing that $F_u\le F_v$.
\end{proof}

\begin{rem}
\rm
Notice that the mean $\mu^{mpr}_p(\ve,u)$ in \eqref{mean-mpr} 
is not monotonic when $1\le p<2$ and is continuous in $L^p(B_\ve(x))$ only for
$p=2, \infty$.
\par
The mean $\mu^{hr,2}_p(\ve,u)$ in \eqref{svr} is not monotonic for $2<p<\infty$ and
the mean $\mu^{hr,1}_p(\ve,u)$ in \eqref{svh} is not continuous unless $p=\infty$.
\par
Finally, the mean $\mu_p^{kmp}(\ve,u)$ in \eqref{mean-kmp} is not continuous unless $p=\infty$. 
\par
To disprove continuity, it is sufficient to take the sequence of
functions $u_n(y)=(|y-x|/\ve)^n$ for $y\in B_\ve(x)$: this converges to zero in $L^p(B_\ve(x))$
for $1\le p<\infty$, but the average of its maximum and minimum is always $1/2$.
\end{rem}
\par
The proof of following proposition is straightforward. 
\begin{prop}
\label{prop:lift&homogeneity}
We have that
\begin{enumerate}[(i)]
\item
$\mu^X_p(u+c)=c+\mu^X_p(u)$ for every $c\in\RR$;
\item
$\mu^X_p(\al u)=\al\,\mu^X_p(u)$ for every $\al\in\RR$.
\end{enumerate}
\end{prop}

\section{The (AMVP) for the elliptic case}
\label{sec:elliptic}

This section is devoted to prove Theorem \ref{th:main:visco}. 
We first give a proof of the (AMVP) for smooth functions. The following lemma is the crucial step of that proof.

\begin{lem}
\label{lm:qp}
Let $1\le p\le\infty$, pick $\xi\in\RR^N\setminus\{0\}$, and let $A$ be a symmetric $N\times N$ matrix. Consider the quadratic function
$q:B_\ve(x)\to\RR$ defined by 
$$
q(y)=q(x)+\xi\cdot(y-x)+\frac12\,\lan A(y-x), y-x\ran, \ y\in B_\ve(x).
$$
\par
Then it holds that
$$
\mu_p(\ve,q)(x)=q(x)+\frac1{2\,(N+p)}\,\left\{ \tr(A)+(p-2)\,\frac{\lan A \xi, \xi\ran}{|\xi|^2}\right\}\,\ve^2+o(\ve^2)
$$
as $\ve\to 0$.
 \end{lem}
\begin{proof}
Set
$$
q_\ve(z)=q(x+\ve\,z), \ \
v_\ve(z)=\frac{q(x+\ve z)-q(x)}{\ve}, \ \mbox{ and } \ v(z)=\xi\cdot z.
$$ 
We know that
$$
\mu_p(\ve, q)(x)=\mu_p(1, q_\ve)(0);
$$ 
thus, Proposition \ref{prop:lift&homogeneity} implies that
$$
\frac{\mu_p(\ve, q)(x)-q(x)}{\ve}=\mu_p(1,v_\ve )(0),
$$
and $v_\ve$ converges to $v$ uniformly on $B$ as $\ve\to 0$. Theorem \ref{thm:continuity} then 
yields that
$$
\lim_{\ve\to 0}\frac{\mu_p(\ve, q)(x)-q(x)}{\ve}=\mu_p(1,v)(0),
$$
and $\la=\mu_p(1,v)(0)=0$, $0$ is the unique root of 
the equation
$$
\int_B|v(z)-\la|^{p-2} [v(z)-\la]\,dz=0,
$$
for $1\le p<\infty$, and for $p=\infty$ maximizes the quantity
$$
\max(|\xi|-\la, \la+|\xi|)
$$
for $\la\in\RR$.
\par
For $1\le p\le\infty$, set
\begin{equation}
\label{def-delta}
\de_\ve=\frac{\mu_p(\ve, q)(x)-q(x)}{\ve^2}.
\end{equation}

\medskip

{\bf Case $\mathbf{1<p<\infty}$.} We define the function  $h(s)=|s|^{p-2} s$;
by some manipulations, we get that
$$
\int_B h(\xi\cdot z+\ve\left[ \lan A z, z\ran/2-\de_\ve\right])\,dz=0.
$$
Without loss of generality, we assume that $|\xi|=1$, apply the change of variables $z=R\,y$, where $R$ is a rotation matrix such that $^tR\,\xi=e_1$, and set $C=^tR\,A\,R$ to obtain that
$$
\int_B \frac{h(y_1+\ve[\lan C y, y\ran/2-\de_\ve])-h(y_1)}{\ve}\,dy=0,
$$
since $\int_B h(y_1)\,dy=0$. Thus, we have that
\begin{equation}
\label{lagrange}
\int_B \left\{\int_0^1 h'(y_1+\tau\,\ve\left[\lan C y, y\ran/2-\de_\ve\right])\,d\tau\right\}\,\left[\lan C y, y\ran/2-\de_\ve\right]\,dy=0,
\end{equation}
and this implies that $\de_\ve$ is bounded by some constant $c$
($c$ is equal to half of the norm of the matrix $C$).
\par
If $2\le p<\infty$, it is easy to prove that, by the dominated convergence theorem, (any converging subsequence of)  $\de_\ve$ converges to the number $\de_0$ defined by
\begin{equation}
\label{def-delta-zero-ell}
\int_B h'(y_1)\,[\lan C y, y\ran/2-\de_0]\,dy=0.
\end{equation}
If $1<p<2$, we observe that 
$$
\left|\int_0^1 h'(y_1+\tau\,\ve\left[\lan C y, y\ran/2-\de_\ve\right])\,\left[\lan C y, y\ran/2-\de_\ve\right]\,d\tau\right|\le
2\,c\,\Bigl||y_1|-2c\,\ve\Bigr|^{p-2}
$$
and
$$
\lim_{\ve\to 0}\int_B \left||y_1|-2 c\,\ve\right|^{p-2}\,dy=\int_B |y_1|^{p-2}\,dy.
$$
If (any converging subsequence of)  $\de_\ve$ converges to a number $\de_0$, 
then the integrand in \eqref{lagrange} converges pointwise to $h'(y_1)\,[\lan C y, y\ran/2-\de_0]$, 
and hence we can conclude that \eqref{def-delta-zero-ell} holds, by the generalized dominated convergence theorem (Theorem \ref{generalized-L}).
\par
Therefore, by Lemma \ref{lem:integrals} we have that
\begin{multline*}
\lim_{\ve\to 0}{\de_\ve}=\frac12\,\frac{\int_B|y_1|^{p-2}\,\lan C y,y\ran\,dy}{\int_B|y_1|^{p-2}\,dy}=\\
\frac1{2\,(N+p)}\,\left\{ \tr(C)+(p-2)\,\lan C\,e_1,e_1\ran\right\}=\\
\frac1{2\,(N+p)}\,\left\{ \tr(A)+(p-2)\,\frac{\lan A \xi, \xi\ran}{|\xi|^2}\right\},
\end{multline*}
since $\lan C\,e_1,e_1\ran=\lan AR\,e_1,R\,e_1\ran$, with $R\,e_1=\xi/|\xi|$.

\medskip

{\bf Case $\mathbf{p=1}$.} We know that $\mu_1(\ve,q)$ is the unique root of the equation
\begin{equation}
\label{def-median}
|\{y\in B_\ve(x): q(y)>\mu_1(\ve,q)\}|=|\{y\in B_\ve(x): q(y)<\mu_1(\ve,q)\}|.
\end{equation}
\par
Next, manipulating \eqref{def-median} gives that
$$
\left|\left\{z\in B: \xi\cdot z+\frac{\ve}{2}\,\lan A z, z\ran>\ve \de_\ve\right\}\right|=
\left|\left\{z\in B: \xi\cdot z+\frac{\ve}{2}\,\lan A z, z\ran<\ve\de_\ve\right\}\right|
$$
and, by applying the substitution $z=R\,y$, where $R$ is a rotation matrix such that $^t R\,\xi=|\xi|\,e_1$, we can infer that
\begin{multline}
\label{volume-identity}
\left|\left\{y\in B: |\xi|\,y_1+\frac{\ve}{2}\,\lan C\,y, y\ran>\ve\de_\ve\right\}\right|=\\
\left|\left\{y\in B: |\xi|\,y_1+\frac{\ve}{2}\,\lan C\,y, y\ran<\ve\de_\ve\right\}\right|,
\end{multline}
where $C=^tR A R$. 
\par
Now, consider the right-hand side of the last formula, set
$$
f_\ve(y)= |\xi|\,y_1+\frac{\ve}{2}\,\lan C\,y, y\ran,
$$
and
$$
c_\ve=\left|\left\{y\in B: f_\ve(y)<\ve\de_\ve\right\}\right|-|B^-|,
$$
where $B^-=\{y\in B: y_1\le 0\}$. 
The use of the change of variables 
$$
y=\frac{\ve\,z_1}{|\xi|}\,e_1+z' \ \mbox{  where } \ z'=(0, z_2,\dots, z_N),
$$
yields that
$$
\frac{|\xi|}{\ve}\,c_\ve=\left|\left\{(z_1,z')\in B^\ve: \ve^{-1}\,f_\ve\left(\frac{\ve\,z_1}{|\xi|}\,e_1+z' \right)<\de_\ve\right\}\right|-|B^{\ve,-}|,
$$
where 
\begin{multline*}
B^\ve=\{ (z_1, z')\in\RR^N: (\ve\,z_1/|\xi|)^2+|z'|^2<1\}, \\ \mbox{ and } \ B^{\ve,-}=\{ z\in B^\ve: z_1\le 0\}.
\end{multline*}
\par
Now, set $B'=\{ z'\in\RR^{N-1}: |z'|<1\}$ and notice that, if $\ve$ is small enough, by the implicit function theorem, there is a
unique function $g_\ve: B'\to\RR$ such that
$$
\ve^{-1}\,f_\ve\left(\frac{\ve\,g_\ve(z')}{|\xi|}\,e_1+z' \right)=\de_\ve \ \mbox{ for } \ z'\in B'.
$$
We can then infer that
\begin{multline*}
c_\ve=\frac{\ve}{|\xi|}\,\int_{B'}\left\{\min\left[g^+_\ve(z'),(|\xi|/\ve)\sqrt{1-|z'|^2}\right]-\right. \\
\left. \min\left[g^-_\ve(z'),(|\xi|/\ve)\sqrt{1-|z'|^2}\right]\right\}\,dz',
\end{multline*}
where $g^+_\ve$ and $g^-_\ve$ denote the positive and negative parts of $g_\ve$.
Thus, since $g_\ve(z')\to \de_0-\lan C z', z'\ran/2$ pointwise (possibly passing to a subsequence), by the dominated convergence theorem, we obtain that
$$
\lim_{\ve\to 0}\frac{|\xi|}{\ve}\,c_\ve=
\int_{B'}\left[\de_0-\frac12\,\lan C z', z'\ran\right]\,dz'.
$$
\par
We can repeat the same arguments for the left-hand side of \eqref{volume-identity} and obtain that
$$
\lim_{\ve\to 0}\frac{|\xi|}{\ve}\,\left\{\left|\left\{y\in B: f_\ve(y)>\ve\de_\ve\right\}\right|-|B^+|\right\}=
\int_{B'}\left[\frac12\,\lan C z', z'\ran-\de_0\right]\,dz'.
$$
Therefore, \eqref{volume-identity} implies that
$$
\int_{B'}\left[\de_0-\frac12\,\lan C\,z', z'\ran\right]\,dz'=0,
$$
and hence 
$$
\de_0\,\frac{\om_{N-1}}{N-1}=\frac12\,\int_{B'}\lan C\,z', z'\ran\,dz'=
\frac{\om_{N-1}}{2\,(N^2-1)}\,\sum_{j=2}^{N}C_{jj}.
$$
Finally, the desired conclusion follows from
$$
\lim_{\ve\to 0}\frac{\mu_1(\ve,q)(x)-q(x)}{\ve^2}=\lim_{\ve\to 0}\de_\ve=\de_0,
$$
where
$$
2\,(N+1)\,\de_0^2=\tr(C)-\lan C\,e_1,e_1\ran=
\tr(A)-\frac{\lan A\,\xi, \xi\ran}{|\xi|^2},
$$
since $\lan C\,e_1,e_1\ran=\lan AR\,e_1,R\,e_1\ran$, with $R\,e_1=\xi/|\xi|$.

\medskip

{\bf Case $\mathbf{p=\infty}$.}  For what we already showed at the beginning of this proof, we know that
\begin{multline*}
\frac{\mu_\infty(\ve,q)(x)-q(x)}{\ve}=\\
\frac12\,\left\{\min_{z\in B}[\xi\cdot z+\ve\,\lan A\,z,z\ran/2]+\max_{z\in B}[\xi\cdot z+\ve\,\lan A\,z,z\ran/2]\right\}.
\end{multline*}
Now, notice that, if $\ve$ is sufficiently small,  the minimum and the maximum are respectively attained at the points $z^+_\ve$ and $z^-_\ve$ on $\PA B$ and
$$
z^\pm_\ve=\pm\frac{\xi+\ve\,A\,z'_\ve}{|\xi+\ve\,A\,z'_\ve|}=\pm\frac{\xi}{|\xi|}+o(\ve),
$$
as $\ve\to 0$.
Thus, we can infer that
$$
\frac{\mu_\infty(\ve,q)-q(x)}{\ve^2}=
\frac{\lan A\,z^+_\ve,z^+_\ve\ran+\lan A\,z^-_\ve,z^-_\ve\ran}{4}+o(1)
$$
and conclude that
$$
\frac{\mu_\infty(\ve,q)-q(x)}{\ve^2}\to\frac12\,\frac{\lan A\,\xi, \xi\ran}{|\xi|^2}
$$
as $\ve\to 0$.
\end{proof}

\begin{thm}[Asymptotics for $\mu_p(\ve,u)$ as $\ve\to 0$]
\label{thm:p=p}
Let $1\le p\le\infty$. Let $\Om\subseteq\RR^N$ be an open set and $x\in\Om$.
\par
If $u\in C^2(\Om)$ with $\nabla u(x)\ne 0$,
then
\begin{equation}
\label{asymptotic-formula}
\mu_p(\ve,u)(x)=u(x)+\frac12\,\frac{\De^n_p u(x)}{N+p}\,\ve^2+o(\ve^2) \ \mbox{ as } \ \ve\to 0.
\end{equation}
\end{thm} 

\begin{proof}
Let $\ve>0$ be such that $\ol{B_\ve(x)}\subset\Om$ and consider the function $q(y)$ in Lemma \ref{lm:qp} with $q(x)=u(x)$, $\xi=\na u(x)$ and $A=\na^2u(x)$; also, notice that
$$
\tr(A)+(p-2)\,\frac{\lan A\,\xi, \xi\ran}{|\xi|^2}=\De_p^nu(x).
$$
\par
Set $u_\ve(z)=u(x+\ve z)$ and $q_\ve(z)=q(x+\ve z)$; since $u\in C^2(\Om)$, then for every $\eta>0$ there exists $\ve_\eta>0$ such that
$$
|u_\ve(z)-q_\ve(z)|<\eta\,\ve^2 \ \mbox{ for every } \ z\in\ol{B} \ \mbox{ and } \ 0<\ve<\ve_\eta.
$$
Thus, since by Proposition \ref{prop:lift&homogeneity} 
$$
\mu_p(\ve, q\pm\eta\ve^2)(x)=
\mu_p(\ve, q)(x)\pm\eta\ve^2,
$$
Theorem \ref{thm:monotonicity} and Corollary \ref{cor:characterization} yield that
$$
\frac{\mu_p(\ve, q)(x)-u(x)}{\ve^2}-\eta\le 
\frac{\mu_p(\ve, u)(x)-u(x)}{\ve^2}\le 
\frac{\mu_p(\ve, q)(x)-u(x)}{\ve^2}+\eta.
$$
Therefore, Lemma \ref{lm:qp} implies that
\begin{multline*}
\frac12\,\frac{\De^n_p u(x)}{N+p} -\eta\le
\liminf_{\ve\to 0} \frac{\mu_p(\ve, u)(x)-u(x)}{\ve^2}\le \\
\limsup_{\ve\to 0} \frac{\mu_p(\ve, u)(x)-u(x)}{\ve^2}\le \frac12\,\frac{\De^n_p u(x)}{N+p} +\eta.
\end{multline*}
The desired conclusion follows, since $\eta$ is arbitrary.
\end{proof}

\begin{cor}[(AMVP) for smooth functions]
\label{th:main} 
Let $1\le p\le\infty$ and $u\in C^2(\Om)$. The following assertions are equivalent:
\begin{enumerate}[(i)]
\item
$\De_p^n u(x)=0$ at any $x\in\Om$ such that $\na u(x)\not=0$;
\item
$u(x)=\mu_p(\ve,u)(x)+o(\ve^2)$ as $\ve\to 0$ at any $x\in\Om$ such that $\na u(x)\not=0$.
\end{enumerate}
\end{cor}

\begin{rem}
\rm 
Without any essential modification, we can show a similar result with $\mu_p(\ve,u)(x)$ in Corollary \ref{th:main} replaced by an analogous spherical $p$-mean value of $u$ on $\PA B_\ve(x)$, that is the 
minimum value in the variational problem \eqref{var-probl}, where the $L^p$ norm is taken on   
$\PA B_\ve(x)$.
The asymptotic formula \eqref{asymptotic-formula} reads in this case as:
\begin{equation*}
\mu_p(\ve,u)(x)=u(x)+\frac12\,\frac{\De^n_p u(x)}{N+p-2}\,\ve^2+o(\ve^2) \ \mbox{ as } \ \ve\to 0.
\end{equation*}
\end{rem}

\medskip

We are now going to prove that continuous viscosity solutions of the normalized $p$-Laplace equation
are characterized by an (AMVP) in the viscosity sense.  We recall the relevant definitions from \cite{MPR2}.

\medskip

A function $u\in C(\Om)$ is a {\it viscosity solution of $\Delta_p^n u=0$ in $\Om$}, if both of the following requisites hold at every $x\in\Om$:
\begin{enumerate}[(i)]
\item
for any function $\phi$ of class $C^2$ near $x$ such that $u-\phi$ has a strict minimum at $x$ with 
$u(x)=\phi(x)$ and $\nabla\phi(x)\ne 0$, there holds that  $\Delta^n_p\phi(x)\leq 0$;
\item 
for any function $\phi$ of class $C^2$ near $x$ such that $u-\phi$ has a strict maximum at $x$ with 
$u(x)=\phi(x)$ and $\nabla\phi(x)\ne 0$, there holds that  $\Delta^n_p\phi(x)\geq 0$. 
\end{enumerate}
\par
We say that a function $u\in C(\Om)$ satisfies at $x\in\Om$ the {\it asymptotic mean value property (AMVP)} 
$$
u(x)=\mu_p(\ve,u)(x)+o(\ve^2) \ \mbox{ as } \ \ve\to 0
$$
{\it in  the viscosity sense}
if both of the following requisites hold:
\begin{enumerate}[(a)]
\item
for any function $\phi$ of class $C^2$ near $x$ such that $u-\phi$ has a strict minimum at $x$ with 
$u(x)=\phi(x)$ and $\nabla\phi(x)\ne 0$, there holds that 
$$
\phi(x)\geq\mu_p(\ve,\phi)(x)+o(\ve^2) \ \mbox{ as } \ \ve\to 0;
$$
\item
for any function $\phi$ of class $C^2$ near $x$ such that $u-\phi$ has a strict maximum at $x$ with 
$u(x)=\phi(x)$ and $\nabla\phi(x)\ne 0$, there holds that 
$$
\phi(x)\leq\mu_p(\ve,\phi)(x)+o(\ve^2) \ \mbox{ as } \ \ve\to 0.
$$
\end{enumerate}
\par

We are now in the position to prove Theorem \ref{th:main:visco}. 

\begin{proof}[Proof of Theorem \ref{th:main:visco}] 
Let $\phi$ be of class $C^2$ near $x$ with $\nabla \phi(x)\ne 0$; by Theorem \ref{thm:p=p}, we know that
\begin{equation}\label{final-phi}
\phi(x)=\mu_p(\ve,\phi)(x)-\frac12\,\frac{\ve^2}{N+p}\Delta_p^n\phi(x)+o(\ve^2)
\end{equation}
as $\ve\to 0$. 
\par
Thus, if $u-\phi$ has a strict minimum at $x$ with $u(x)=\phi(x)$ and $\De_p^n \phi(x)\le 0$, then 
\eqref{final-phi} implies that 
$$
\phi(x)\ge\mu_p(\ve,\phi)(x)+o(\ve^2) \ \mbox{ as } \ \ve\to 0.
$$
Conversely, if $\phi(x)\ge\mu_p(\ve,\phi)(x)+o(\ve^2)$ as $\ve\to 0$, by \eqref{final-phi} we infer that
$$
-\De_p^n\phi(x)\ge o(1) \ \mbox{ as } \ \ve\to 0,
$$
and hence $\De_p^n\phi(x)\le 0$.
\par
We proceed similarly, if $u-\phi$ has a strict maximum at $x$.
\end{proof}

\section{The (AMVP) for the parabolic case}
\label{sec:parabolic}

The situation in the parabolic case is similar to that presented in the previous paragraph: we just
have to use the proper cost function. As already observed, the choice disclosed in \eqref{var-probl-heat} is a good candidate since it yields for $p=2$ the classical mean value property for solutions of the heat equation. Thus, we shall denote: 
\begin{equation}
\label{def-pi}
\mbox{$\pi_p(\ve, u)(x,t)=$ the unique $\pi\in\RR$ satisfying \eqref{var-probl-heat}.}
\end{equation}
\par
It is clear that the characterization, continuity and monotonicity of Theorems \ref{thm:characterization}, \ref{thm:continuity} and \ref{thm:monotonicity}
apply to $\pi_p(\ve, u)(x,t)$, if we set $X=\ol{E_\ve(x,t)}$ and $d\nu(y,s)=|x-y|^2/(t-s)^2 dy\, ds$. In particular,
the {\it heat mean value of $u$} is 
$$
\dashint_{E_\ve(x,t)}u(y,s)\,d\nu(y,s)=\frac1{4\,\ve^{N}}\int_{E_\ve(x,t)} u(y,s)\,d\nu(y,s)
$$ 
and the {\it heat median of $u$}, $\pmed\limits_{E_\ve(x,t)} u$, is the unique root of the equation:
\begin{equation}
\label{def-median-parabolic}
\int_{E_\ve^{\la,+}(x,t)}\,\frac{|x-y|^2}{(t-s)^2}\,dy\,ds=\int_{E_\ve^{\la,-}(x,t)}\,\frac{|x-y|^2}{(t-s)^2}\,dy\,ds,
\end{equation}
where
$$
E_\ve^{\la,\pm}(x,t)=\{(y,s)\in E_\ve(x,t): \la\lessgtr u(y,s)\}.
$$
\par
The companion of Corollary \ref{cor:characterization} is the following result, that does not need an ad hoc proof.
\begin{cor}
\label{thm:characterization-parabolic}
Let $1\le p<\infty$, $u\in C(\ol{E_\ve(x,t)})$ and define
$$
u_\ve(z,\si)=u(x+\ve z, t-\ve^2 \si), \ (z,\si)\in E,
$$
where
\begin{equation}
\label{def-E}
E=\{ (z,\si)\in\RR^{N+1} : 0<\si<\frac1{4\pi}, \Phi(z,\si)>1\}.
\end{equation}
\par
Then
\begin{equation}
\label{riscale-parabolic}
\pi_p(\ve, u)(x,t)=\pi_p(1, u_\ve)(0,0),
\end{equation}
where $\la=\pi_p(1, u_\ve)(0,0)$ is the unique root of the equation
\begin{equation}
\label{characterization-parabolic}
\int_E |u_\ve(z,\si)-\la|^{p-2}[u_\ve(z,\si)-\la]\,d\nu(z,\si)=0.
\end{equation}
\end{cor}

\begin{lem}
\label{lm:qp-parabolic}
Let $1\le p\le\infty$, pick $a\in\RR$ and $\xi\in\RR^N\setminus\{0\}$, and let $A$ be a symmetric $N\times N$ matrix. 
\par
Consider the quadratic function
$q:E_\ve(x,t)\to\RR$ defined by 
$$
q(y,s)=q(x,t)+\xi\cdot(y-x)+a\,(s-t)+\frac12\,\lan A(y-x), y-x\ran
$$
for $(y,s)\in E_\ve(x,t)$.
Let $\pi_p(\ve,q)$ be the heat $p$-mean of $q$ on $E_\ve(x,t)$.
\par
Then it holds that
\begin{multline*}
\pi_p(\ve,q)=q(x,t)+\\
\frac1{4\pi}\left(1-\frac{2}{N+p}\right)^{1+\frac{N+p}{2}}\!\!
\left\{-a+\frac{N}{N+p-2}\left[\tr(A)+(p-2)\frac{\lan A \xi, \xi\ran}{|\xi|^2}\right]\right\}\ve^2+o(\ve^2)
\end{multline*}
as $\ve\to 0$.
 \end{lem}

 \begin{proof}
We proceed similarly to the proof of Lemma \ref{lm:qp}.
Set
\begin{multline*}
q_\ve(z, \si)=q(x+\ve\,z, t-\ve^2\,\si), \\
v_\ve(z, \si)=\frac{q(x+\ve z, t-\ve^2\,\si)-q(x,t)}{\ve} \ \mbox{ and } \ v(z,\si)=\xi\cdot z.
\end{multline*}
We know that
$$
\pi_p(\ve, q)(x,t)=\pi_p(1, q_\ve)(0,0);
$$
thus, Proposition \ref{prop:lift&homogeneity} implies that
$$
\frac{\pi_p(\ve, q)(x,t)-q(x,t)}{\ve}=\pi_p(1,v_\ve )(0,0),
$$
and $v_\ve$ converges to $v$ uniformly on $E$ as $\ve\to 0$. Theorem \ref{thm:continuity} then 
yields that
$$
\lim_{\ve\to 0}\frac{\pi_p(\ve, q)(x,t)-q(x,t)}{\ve}=\pi_p(1,v)(0,0),
$$
and $\pi_p(1, v)=0$, since it is the unique solution $\la$ of
$$
\int_E|v(z, \si)-\la|^{p-2} [v(z, \si)-\la]\, d\nu(z,\si)=0,
$$
for $1\le p<\infty$, and for $p=\infty$ maximizes the quantity
$$
\max(|\xi|-\la, \la+|\xi|)
$$
for $\la\in\RR$.
\par
As before, set
$$
\de_\ve=\frac{\pi_p(\ve, q)(x,t)-q(x,t)}{\ve^2}.
$$
\medskip

{\bf Case $\mathbf{1<p<\infty}$.} 
By some manipulations, we get that
$$
\int_E h(\xi\cdot z+\ve\left[ -a\,\si+\lan A z, z\ran/2-\de_\ve\right])\, d\nu(z,\si)=0,
$$
where $h$ is the function already defined.
Without loss of generality, we assume that $|\xi|=1$, apply the change of variables $z=R\,y$, where $R$ is a rotation matrix such that $^tR\,\xi=e_1$, and set $C=^tR\,A\,R$ to obtain that
$$
\int_E \frac{h(y_1+\ve[-a\,\si+\lan C y, y\ran/2-\de_\ve])-h(y_1)}{\ve}\, d\nu(y,\si),
$$
since $\int_E h(y_1)\,|y|^2/\si^2\, dy d\si=0$. Thus, by proceeding as before, we have that
\begin{multline}
\label{lagrange-parabolic}
\de_\ve\,\int_E \left\{\int_0^1 h'(y_1+\tau\,\ve\left[-a\,\si+\lan C y, y\ran/2-\de_\ve\right])\,d\tau\right\} d\nu(y,\si)=\\
\int_E \left\{\int_0^1 h'(y_1+\tau\,\ve\left[-a\,\si+\lan C y, y\ran/2-\de_\ve\right])\,d\tau\right\}\,\left[-a\,\si+\lan C y, y\ran/2\right] d\nu(y,\si),
\end{multline}
and this implies that $\de_\ve$ is bounded by some constant 
(this is equal to $c+|a|/4\pi$).
\par
If $2\le p<\infty$, it is easy to prove that, by the dominated convergence theorem, (any converging subsequence of)  $\de_\ve$ converges to the number $\de_0$ defined by
\begin{equation}
\label{def-delta-zero}
\int_E h'(y_1)\,[-a\,\si+\lan C y, y\ran/2-\de_0]\, d\nu(y,\si)=0.
\end{equation}
If $1< p<2$, we observe that 
\begin{multline*}
\left|\int_0^1 h'(y_1+\tau\,\ve\left[-a\,\si+\lan C y, y\ran/2-\de_\ve\right])
\,\left[-a\,\si+\lan C y, y\ran/2-\de_\ve\right]\,d\tau\right| 
\le \\
2\,(c+|a|/4\pi)\,\Bigl||y_1|-2\,(c+|a|/4\pi)\,\ve\Bigr|^{p-2}
\end{multline*}
and
$$
\lim_{\ve\to 0}\int_E \Bigl||y_1|-2\,(c+|a|/4\pi)\,\ve\Bigr|^{p-2}\,d\nu(y,\si)=\int_B |y_1|^{p-2}\,d\nu(y,\si).
$$
If (any converging subsequence of)  $\de_\ve$ converges to a number $\de_0$, then the integrand in \eqref{lagrange} converges pointwise to $h'(y_1)\,[-a\,\si+\lan C y, y\ran/2-\de_0]$, and hence we can conclude that \eqref{def-delta-zero} holds, by the generalized dominated convergence theorem (Theorem \ref{generalized-L}).
\par
Therefore, by Lemma \ref{lem:integrals-parabolic} we have that
\begin{multline*}
\lim_{\ve\to 0}{\de_\ve}=\frac{\int_E|y_1|^{p-2}\,[-a\,\si+\lan C y,y\ran/2]\,d\nu(y,\si)}{\int_E|y_1|^{p-2}\,d\nu(y,\si)}=\\
\frac1{4\pi}\,\left(1-\frac2{N+p}\right)^{\frac{N+p}{2}+1}
\left\{-a+\frac{N}{N+p-2}\,\left[\tr(C)+(p-2)\,\lan C e_1, e_1\ran\right]\right\}=\\
\frac1{4\pi}\,\left(1-\frac2{N+p}\right)^{\frac{N+p}{2}+1}
\left\{-a+\frac{N}{N+p-2}\,\left[\tr(A)+(p-2)\,\frac{\lan A \xi, \xi\ran}{|\xi|^2}\right]\right\}
\end{multline*}
since $\lan C\,e_1,e_1\ran=\lan AR\,e_1,R\,e_1\ran$, with $R\,e_1=\xi/|\xi|$.

\medskip

{\bf Case $\mathbf{p=1}$.}
By proceeding as in the proof of Lemma \ref{lm:qp}, it is easy to show that
$$
\int_{E^+_\ve} \frac{|y|^2}{\si^2}\,dy\,d\si=\int_{E^-_\ve} \frac{|y|^2}{\si^2}\,dy\,d\si
$$
where
$$
E^\mp_\ve=\left\{(y,\si)\in E: -y_1\lessgtr 0, |\xi|\,y_1+\ve\,[-a\,\si+\lan C y, y\ran/2]\lessgtr\ve\de_\ve\right\},
$$
$R$ is the usual rotation matrix, and $C={^tR} A R$. 
\par
Now, we assume that $|\xi|=1$ without loss of generality and use the change of variables 
$$
y=\ve\,z_1\,e_1+z' \ \mbox{  where } \ z'=(0, z_2,\dots, z_N)
$$
and take the limit as $\ve\to 0$; similarly to the proof of Lemma \ref{lm:qp}
we obtain that
$$
\int_{E^+_0} \frac{|z|^2}{\si^2}\,dz\,d\si=\int_{E^-_0} \frac{|z|^2}{\si^2}\,dz\,d\si
$$
where $\de_0$ is, as usual, the limit of $\de_\ve$ as $\de\to 0$ and
\begin{multline*}
\label{def-delta-parabolic}
E^\pm_0=\left\{(z,\si)\in \RR^{N+1}: 0\le |z'|<\sqrt{-2N\,\si\,\log(4\pi \si)},\, 0<\si<\frac1{4\pi},\right. \\ 
\left.  0 \lessgtr z_1 , \de_0 \lessgtr  z_1-a\,\si+\frac{\lan C\,z', z'\ran}{2}\right\}.
\end{multline*}
\par
Thus, $\de_0$ results to be the solution of
$$
\int_{E'}\left[\de_0+a\,\si-\frac12\,\lan C\,z', z'\ran\right]\,\frac{|z'|^2}{\si^2}\,dz'\,d\si=0,
$$
where
$$
E'=\left\{(z',\si)\in \RR^{N}: 0\le |z'|<\sqrt{-2N\,\si\,\log(4\pi \si)},\, 0<\si<\frac1{4\pi}\right\},
$$
and hence 
\begin{multline*}
\de_0\,\int_{E^*}r^N \si^{-2}\,dr\,d\si=-a\,\int_{E^*}r^N \si^{-1}\,dr\,d\si+\\
\frac1{2\,(N-1)}\,\left[\tr(A)-\frac{\lan A\,\xi, \xi\ran}{|\xi|^2}\right]\,\int_{E^*}r^{N+2} \si^{-2}\,dr\,d\si.
\end{multline*}
Finally, Lemma \ref{lem:numbers-parabolic} gives that
$$
\de_0=\frac1{4\pi}\,\left(\frac{N-1}{N+1}\right)^{\frac{N+1}{2}+1}
\left\{-a+\frac{N}{N-1}\,\left[\tr(A)-\frac{\lan A\,\xi, \xi\ran}{|\xi|^2}\right]\right\}.
$$

\medskip

{\bf Case $\mathbf{p=\infty}$.}
For what we already showed at the beginning of this proof, we know that
\begin{multline*}
\frac{\pi_\infty(\ve,q)-q(x,t)}{\ve}=\frac12\min_{(z,\si)\in E}\left[\xi\cdot z+\ve\,(-a\,\si+\lan A\,z,z\ran/2)\right]+\\
\frac12\max_{(z,\si)\in E}\left[\xi\cdot z+\ve\,(-a\,\si+\lan A\,z,z\ran/2)\right].
\end{multline*}
\par
Now, notice that if $\ve$ is sufficiently small,  since $\xi\not=0$, the minimum and the maximum are attained at some points $(z^+_\ve, \si^+_\ve)$ and $(z^-_\ve, \si^-_\ve)$ on $\PA E$. Thus, there exist two Lagrange multipliers $\la^+_\ve$ and $\la^-_\ve$ such that the 
following three equations hold:
\begin{equation}
\label{lagrange-multipliers}
\begin{array}{cc}
&\xi+\ve\,A z^\pm_\ve=\la^\pm_\ve\,z^\pm_\ve, \quad -\ve\,a=\la^\pm_\ve N \{\log(4\pi\si^\pm_\ve)+1\}, \vspace{5pt} \\ 
&|z^\pm_\ve|^2+2N\si^\pm_\ve\log(4\pi\si^\pm_\ve)=0.
\end{array}
\end{equation}
\par
Since
$$
\max_{(z,\si)\in E}(\xi\cdot z)=|\xi|\,\sqrt{\frac{N}{2\pi e}} \ \mbox{ and } \ \min_{(z,\si)\in E}(\xi\cdot z)=- |\xi|\,\sqrt{\frac{N}{2\pi e}},
$$
a straightforward asymptotic analysis on the system \eqref{lagrange-multipliers} informs us that
$$
z^\pm_\ve=\pm\sqrt{\frac{N}{2\pi e}}\frac{\xi}{|\xi|}+o(\ve) \ \mbox{ and } \ 
\si^\pm_\ve=\frac1{4\pi e}+o(\ve) \ \mbox{ as } \ \ve\to 0.
$$
Therefore, we obtain:
\begin{multline*}
\lim_{\ve\to 0}\frac{\pi_\infty(\ve,q)-q(x,t)}{\ve^2}=\\
\lim_{\ve\to 0}\left\{\xi\cdot\frac{z^-_\ve+z^+_\ve}{2\ve}-a\, \frac{\si^-_\ve+\si^+_\ve}{2}+\frac{\lan A z^-_\ve, z^-_\ve\ran+\lan A z^+_\ve, z^+_\ve\ran}{4}\right\}=\\
\frac1{4\pi e}\left(-a+N\,\frac{\lan A \xi, \xi\ran}{|\xi|^2}\right),
\end{multline*}
as desired.
\end{proof}

\begin{thm}[Asymptotics for $\pi_p(\ve,u)$ as $\ve\to 0$]
Let $1\le p\le \infty$. Assume $(x,t)\in\Om_T$, $u\in C^2(\Om_T)$ and $\nabla u(x,t)\ne 0$.
\par
Then
\begin{multline}
\pi_p(\ve,u)(x,t)=u(x,t)+\\
\frac1{4\pi}\,\left(1-\frac2{N+p}\right)^{\frac{N+p}{2}+1}
\left\{-u_t(x,t)+\frac{N}{N+p-2}\,\De^n_p u(x,t)\right\}\,\ve^2+
o(\ve^2), 
\end{multline}
 as $\ve\to 0$.
\end{thm} 

\begin{proof}
Let $\ve>0$ be such that $\ol{E_\ve(x,t)}\subset\Om_T$ and consider the function $q(y,s)$ in Lemma \ref{lm:qp-parabolic} with $q(x,t)=u(x,t)$, $a=u_t(x,t)$, $\xi=\na u(x,t)$, and $A=\na^2u(x,t)$; then, set $u_\ve(z,\si)=u(x+\ve z, t-\ve^2 \si)$ and $q_\ve(z,\si)=u(x+\ve z, t-\ve^2 \si)$. 
\par
Since $u\in C^2(\Om_T)$, for every $\eta>0$ there exists $\ve_\eta>0$ such that
$$
|u_\ve(z,\si)-q_\ve(z,\si)|<\eta\,\ve^2 \ \mbox{ for every } \ z\in\ol{E} \ \mbox{ and } \ 0<\ve<\ve_\eta.
$$
Thus, by Proposition \ref{prop:lift&homogeneity},  \eqref{riscale-parabolic} and Theorem \ref{thm:monotonicity},
\begin{multline*}
\frac{\pi_p(\ve, q)(x,t)-q(x,t)}{\ve^2}-\eta\le 
\frac{\pi_p(\ve, u)(x,t)-q(x,t)}{\ve^2}\le \\
\frac{\pi_p(\ve, q)(x,t)-q(x,t)}{\ve^2}+\eta.
\end{multline*}
Therefore, Lemma \ref{lm:qp-parabolic} implies that
\begin{multline*}
\frac1{4\pi}\,\left(1-\frac2{N+p}\right)^{\frac{N+p}{2}+1}
\left\{-u_t(x,t)+\frac{N}{N+p-2}\,\De^n_p u(x,t)\right\} -\eta\le \\
\liminf_{\ve\to 0} \frac{\pi_p(\ve, u)(x,t)-q(x,t)}{\ve^2}\le 
\limsup_{\ve\to 0} \frac{\pi_p(\ve, u)(x,t)-q(x,t)}{\ve^2}\le \\
\frac1{4\pi}\,\left(1-\frac2{N+p}\right)^{\frac{N+p}{2}+1}
\left\{-u_t(x,t)+\frac{N}{N+p-2}\,\De^n_p u(x,t)\right\} +\eta.
\end{multline*}
The desired conclusion follows at once, since $\eta$ is arbitrary.
\end{proof}

\begin{cor}
\label{cor:}
Let $u\in C^2(\Om_T)$. The following assertions are equivalent:
\begin{enumerate}[(i)]
\item
$-u_t(x,t)+\frac{N}{N+p-2}\,\De_p^n u(x,t)=0$,
\item
$u(x)=\pi_p(\ve,u)(x,t)+o(\ve^2)$ as $\ve\to 0$,
\end{enumerate}
at any point $(x,t)\in\Om_T$ such that $\na u(x,t)\not=0$.
\end{cor}

We do not provide the proof of Theorem \ref{th:pa-visco}, since is a straightforward re-adaptation
of that of Theorem \ref{th:main:visco}, once the following definitions are established. 

\medskip

A function $u\in C(\Om_T)$ is a {\it viscosity solution of $u_t=\frac{N}{N+p-2}\Delta_p^n u$ in $\Om_T$}, if both of the following requisites hold at every $(x,t)\in\Om_T$:
\begin{enumerate}[(i)]
\item
for any function $\phi$ of class $C^2$ near $(x,t)$ such that $u-\phi$ has a strict minimum at $(x,t)$ with 
$u(x,t)=\phi(x,t)$ and $\nabla\phi(x,t)\ne 0$, there holds that  $\frac{N}{N+p-2}\Delta_p^n u(x,t)\leq u_t(x,t)$;
\item 
for any function $\phi$ of class $C^2$ near $(x,t)$ such that $u-\phi$ has a strict maximum at $(x,t)$ with 
$u(x,t)=\phi(x,t)$ and $\nabla\phi(x,t)\ne 0$, there holds that  $\frac{N}{N+p-2}\Delta_p^n u(x,t)\geq u_t(x,t)$. 
\end{enumerate}
\par
We say that a function $u\in C(\Om_T)$ satisfies at $(x,t)\in\Om_T$ the {\it asymptotic mean value property (AMVP)} 
$$
u(x,t)=\pi_p(\ve,u)(x,t)+o(\ve^2) \ \mbox{ as } \ \ve\to 0
$$
{\it in  the viscosity sense}
if both of the following requisites hold:
\begin{enumerate}[(a)]
\item
for any function $\phi$ of class $C^2$ near $(x,t)$ such that $u-\phi$ has a strict minimum at $(x,t)$ with 
$u(x,t)=\phi(x,t)$ and $\nabla\phi(x,t)\ne 0$, there holds that 
$$
\phi(x,t)\geq\pi_p(\ve,\phi)(x,t)+o(\ve^2) \ \mbox{ as } \ \ve\to 0;
$$
\item
for any function $\phi$ of class $C^2$ near $(x,t)$ such that $u-\phi$ has a strict maximum at $(x,t)$ with 
$u(x,t)=\phi(x,t)$ and $\nabla\phi(x,t)\ne 0$, there holds that 
$$
\phi(x,t)\leq\pi_p(\ve,\phi)(x,t)+o(\ve^2) \ \mbox{ as } \ \ve\to 0.
$$
\end{enumerate}

\section{Useful integrals}
\label{sec:computations}

We begin with the computation of some useful integrals.

\begin{lem}
\label{lem:integrals}
Let $\SS^{N-1}$ be the unit sphere in $\RR^N$. Let $\xi\in\RR^N\setminus\{0\}$ and $A$ be an $N\times N$ symmetric matrix. Then for $1<p<\infty$ we have that
\begin{equation}
\label{int1}
\frac{\int_{\SS^{N-1}} |\xi\cdot y|^{p-2}\lan A y, y\ran\,dS_y}{\int_{\SS^{N-1}} |\xi\cdot y|^{p-2}\,dS_y}=
\frac1{N+p-2}\left\{\tr(A)+(p-2)\,\frac{\lan A \xi, \xi\ran}{|\xi|^2}\right\}
\end{equation}
and
\begin{equation}
\label{int2}
\frac{\int_B |\xi\cdot y|^{p-2}\lan A y, y\ran\,dy}{\int_B |\xi\cdot y|^{p-2}\,dy}=
\frac1{N+p}\left\{\tr(A)+(p-2)\,\frac{\lan A \xi, \xi\ran}{|\xi|^2}\right\}.
\end{equation}
\end{lem}
\begin{proof}
Let $R$ be a rotation matrix such that $^t R \xi=|\xi|\,e_1$; by the change of variables $y= R \te$, we have that
$$
\frac{\int_{\SS^{N-1}} |\xi\cdot y|^{p-2}\lan A y, y\ran\,dS_y}{\int_{\SS^{N-1}} |\xi\cdot y|^{p-2}\,dS_y}=\frac{\int_{\SS^{N-1}} |\te_1|^{p-2}\lan (^t R A R)\,\te, \te\ran\,dS_\te}{\int_{\SS^{N-1}} |\te_1|^{p-2}\,dS_\te}.
$$
On the other hand, 
\begin{multline*}
\int_{\SS^{N-1}} |\te_1|^{p-2} \te_i\,\te_j\,dS_\te=
\int_B \frac{\PA}{\PA y_j}(|y_1|^{p-2} y_i)\,dy=\\
[\delta_{ij}+(p-2) \delta_{i1} \delta_{1j}]
\int_B |y_1|^{p-2} dy=\frac{\delta_{ij}+(p-2) \delta_{i1} \delta_{1j}}{N+p-2}\,
\int_{\SS^{N-1}} |\te_1|^{p-2}\,dS_\te,
\end{multline*}
where we have used the divergence theorem in the first equality. Therefore, we obtain that
\begin{multline*}
\frac{\int_{\SS^{N-1}} |\xi\cdot y|^{p-2}\lan A y, y\ran\,dS_y}{\int_{\SS^{N-1}} |\xi\cdot y|^{p-2}\,dS_y}=\\
\frac{\tr(^t R A R)+(p-2)\lan (^t R A R)\,e_1, e_1\ran}{N+p-2}=\\
\frac1{N+p-2}\,\left\{\tr(A)+(p-2)\,\frac{\lan A\,\xi, \xi\ran}{|\xi|^2}\right\}.
\end{multline*}
\par
Formula \eqref{int2} easily follows from \eqref{int1}.
\end{proof}

\begin{lem}
\label{lem:numbers-parabolic}
Let $\al>0$ and $\be<\al+1$ be real numbers and let
$$
E_*=\left\{(r,\si)\in\RR^2: 0<r<\sqrt{-2N\,\si\,\log(4\pi\si)},\,0<\si<\frac1{4\pi}\right\}.
$$
Then
\begin{equation}
\label{C-alpha-beta}
\int_{E_*} r^{2\al-1} \si^{-\be}\,dr d\si=\frac{2^{2\be-\al-3}\,\pi^{\be-\al-1}\,N^\al}{\al\,(\al-\be+1)^{\al+1}}\,\Ga(\al+1).
\end{equation}
\end{lem}
\begin{proof}
The result follows from the calculations:
\begin{multline*}
\int_{E_*} r^{2\al-1} \si^{-\be}\,dr d\si=\frac1{2\al}\,\int_0^{\frac1{4\pi}}\si^{-\be}\{-2N\si\,\log(4\pi\si)\}^\al\,d\si=\\
\frac{2^{2\be-\al-3}\,\pi^{\be-\al-1}\,N^\al}{\al}\,\int_0^\infty\tau^\al\,e^{-(\al-\be+1)\tau}\,d\tau=\\
\frac{2^{2\be-\al-3}\,\pi^{\be-\al-1}\,N^\al}{\al\,(\al-\be+1)^{\al+1}}\,\int_0^\infty\tau^\al\,e^{-\tau}\,d\tau;
\end{multline*}
in the second equality we used the substitution $4\pi\si=e^{-\tau}$.
\end{proof}
\begin{lem}
\label{lem:integrals-parabolic}
Let $\xi$ and $A$ be as in Lemma \ref{lem:integrals}. Then for $1<p<\infty$ we have that
\begin{equation}
\label{int1-parabolic}
\frac{\int_E |\xi\cdot z|^{p-2}\,\si\,d\nu(z,\si)}{\int_E |\xi\cdot z|^{p-2}\,d\nu(z,\si)}=
\frac1{4\pi}\,\left(\frac{N+p-2}{N+p}\right)^{1+\frac{N+p}{2}}
\end{equation}
and
\begin{multline}
\label{int2-parabolic}
\frac{\int_E |\xi\cdot z|^{p-2}\lan A z, z\ran\,d\nu(z,\si)}{\int_E |\xi\cdot z|^{p-2}\,d\nu(z,\si)}=\\
\frac1{2\pi}\,\frac{N}{N+p-2}\,\left(\frac{N+p-2}{N+p}\right)^{1+\frac{N+p}{2}}\,\left\{\tr(A)+(p-2)\,\frac{\lan A \xi, \xi\ran}{|\xi|^2}\right\}.
\end{multline}
\end{lem}
\begin{proof}
By using  spherical coordinates, we calculate that
$$
\frac{\int_E |\xi\cdot z|^{p-2}\,\si\,d\nu(z,\si)}{\int_E |\xi\cdot z|^{p-2}\,d\nu(z,\si)}=
\frac{\int_{E_*}r^{p+N-1}\,\si^{-1} dr\,d\si}{\int_{E_*}r^{p+N-1}\,\si^{-2} dr\,d\si}
$$
and
$$
\frac{\int_E |\xi\cdot z|^{p-2}\lan A z, z\ran\,d\nu(z,\si)}{\int_E |\xi\cdot z|^{p-2}\,d\nu(z,\si)}=
\frac{\int_{E_*}r^{p+N+1}\,\si^{-2} dr\,d\si}{\int_{E_*}r^{p+N-1}\,\si^{-2} dr\,d\si}\,
\frac{\int_{\SS^{N-1}} |\xi\cdot y|^{p-2}\lan A y, y\ran\,dS_y}{\int_{\SS^{N-1}} |\xi\cdot y|^{p-2}\,dS_y}.
$$
Thus, \eqref{int1-parabolic} and \eqref{int2-parabolic} follow from the calculations
$$
\frac{\int_{E_*}r^{p+N-1}\,\si^{-1} dr\,d\si}{\int_{E_*}r^{p+N-1}\,\si^{-2} dr\,d\si}
$$
and \eqref{int1}.
\end{proof}
\par
For the reader's convenience, we recall the generalized dominated convergence theorem (see \cite{LL} for instance), that is needed for the proofs of Lemmas \ref{lm:qp} and \ref{lm:qp-parabolic}. 

\begin{thm}[Generalized dominated convergence theorem]
\label{generalized-L}
Let $(X,\nu)$ be a measure space and let $\{f_n\}_{n\in\mathbb N}$ 
and $\{g_n\}_{n\in\Bbb N}$ be sequences of measurable functions on $X$ such that
\begin{enumerate}[(i)]
\item
$f_n$ converges to a measurable function $f$ a.e. on $X$ as $n\to\infty$;
\item
each $g_n\in L^1(X, \nu)$ 
and $g_n$ converges to a  function $g$ in $ L^1(X, \nu)$ a.e. on $X$ as $n\to\infty$;
\item
$|f_n|\leq g_n$ a.e. on $X$ for all $n\in\mathbb{N}$;
\item
$\lim\limits_{n\to\infty}\int_X g_n\,d\nu=\int_X g\,d\nu$. 
\end{enumerate}

Then, we have that  
$$
\lim_{n\to\infty}\int_X f_n\,d\nu=\int_X f\,d\nu.
$$ 
\end{thm}

\bigskip
\footnotesize
\noindent\textit{Acknowledgments.}
The second author was supported by a PRIN grant of the 
italian MIUR and the Gruppo Nazionale per l'Analisi Matematica, la Probabilit\`a e le 
loro Applicazioni (GNAMPA) of the Italian Istituto Nazionale di Alta Matematica (INdAM).

\end{document}